\documentclass[letterpaper, 10 pt, conference]{ieeeconf}  

\IEEEoverridecommandlockouts                              

\usepackage{graphicx}
\usepackage{graphics} 
\usepackage{times} 
\usepackage{amsmath} 
\usepackage{amssymb} 
\usepackage{color}
\let\labelindent\relax
\usepackage{enumitem}
\usepackage{cite}
\usepackage{url}

\newtheorem{theorem}{\bf Theorem}

\newtheorem{claim}{\bf Claim}

\newtheorem{lemma}{\bf Lemma}

\newtheorem{conjecture}{\bf Conjecture}
\def\QED{~\rule[-1pt]{5pt}{5pt}\par\medskip}

\newcommand{\bx}{X}
\newcommand{\by}{Y}
\newcommand{\bz}{Z}
\newcommand{\bv}{V}
\newcommand{\bw}{W}

\title{\LARGE \bf
Optimal Sensor Design and Zero-Delay Source Coding for Continuous-Time Vector Gauss-Markov Processes
}

\author{Takashi Tanaka$^{1}$ \and Mikael Skoglund$^{2}$ \and Valeri Ugrinovskii$^{3}$
\thanks{
$^{1}$Department of Aerospace Engineering \& Engineering Mechanics, University of Texas at Austin, Austin, TX, USA. {\tt\small ttanaka@utexas.edu}. 
$^{2}$School of Electrical Engineering, KTH Royal Institute of Technology, Stockholm, Sweden. {\tt\small skoglund@kth.se}. 
$^{3}$School of Engineering and IT, The University of New South Wales at ADFA, Canberra, Australia. {\tt\small v.ougrinovski@adfa.edu.au}.}
}

\begin{document}

\maketitle
\thispagestyle{empty}
\pagestyle{empty}

\begin{abstract}
  We consider the situation in which a continuous-time vector
  Gauss--Markov process is observed through a vector Gaussian channel
  (sensor) and estimated by the Kalman--Bucy filter.  Unlike in
  standard filtering problems where a sensor model is given a
  priori, we are concerned with the optimal sensor design by which (i)
  the mutual information between the source random process and the
  reproduction (estimation) process is minimized, and (ii) the minimum
  mean-square estimation error meets a given distortion constraint.
  We show that such a sensor design problem is tractable by
  semidefinite programming. The connection to zero-delay source-coding
  is also discussed.
\end{abstract}

\section{Introduction}

In this paper, we consider a situation in which a continuous-time
vector Gauss--Markov process (the \emph{source} random process) is
estimated by the Kalman--Bucy filter based on the output of a
memoryless vector Gaussian channel (the \emph{sensor}).  We study this
estimation mechanism from the perspectives of (i) the mean-square
error (MSE) between the source and the estimate, and (ii) the mutual
information rate between the source and the estimate.  From
 standard rate--distortion theory, it is intuitively
clear that an accurate sensing mechanism should lead to a small MSE and a
large mutual information, while a noisy sensing mechanism implies a
large MSE and a small mutual information.  In this paper, we make this
intuition explicit by deriving a trade-off curve between these two
metrics by constructing trade-off achieving sensor gain matrices.  In
particular, we show that trade-off achieving sensor gain matrices are
easily computed by semidefinite programming, and consequently the
trade-off curve admits a convenient semidefinite representation.

There is a simple and explicit relationship (often called the \emph{I-MMSE relationship} in the literature) between the mutual information (I) and the minimum mean-square error (MMSE) when a random variable is observed through a Gaussian channel.
Guo et al. \cite{guo2005mutual} showed that the derivative of the mutual information with respect to the channel SNR (signal-to-noise ratio) is equal to half the MMSE.
They also considered causal estimation of random processes through Gaussian channels and provided a remarkably simple connection between causal and non-causal MMSE.
For continuous-time source processes observed through Gaussian channels, Duncan \cite{duncan1970calculation} already derived a relevant result, stating that ``twice the mutual information is merely the integration of the trace of the optimal mean square filtering error.''
Kadota et al. \cite{kadota1971mutual} considered estimation of continuous-time source over Gaussian channel with feedback (the source is causally affected by channel output).
Weissman et al. \cite{weissman2013directed} further studied the cases with feedback, where a fundamental relationship between directed information and MMSE is derived. 

In parallel with the I-MMSE formulas for Gaussian observations, there exists a line of research for random processes observed through Poisson channels. 
Guo et al. \cite{guo2008mutual}  studied a relationship between mutual information and the estimation error, measured by the mean value of the logarithm of the ratio of the channel input plus dark current and its mean estimate.
Remarkably, the same formula as the I-MMSE relationship for Gaussian case is recovered for Poisson cases as well, provided that MMSE is replaced by a suitable loss function for Poisson channels \cite{atar2012mutual}. 
Estimation of continuous-time processes through Poisson channels with feedback is studied by \cite{kabanov1978capacity,weissman2013directed}.
Recently, an overarching theory unifying the I-MMSE relationship for Gaussian channels and the similar relationship for Poisson channels is proposed \cite{jiao2016mutual}.

Applications of the I-MMSE formula can be found in channel coding problems.
Palomar and Verd\'u \cite{palomar2006gradient} extended the result by \cite{guo2005mutual} to vector Gaussian channels, where an explicit formula relating gradients of mutual information with respect to channel parameters and estimate covariance matrices is obtained.
Based on this result, they proposed a gradient ascent algorithm for channel precoder design where input-output mutual information is maximized subject to input power constraints.

In this paper, we apply Duncan's I-MMSE formula for the aforementioned
trade-off study. Our study is motivated by the zero-delay source
coding problem.  Derpich and {\O}stergaard \cite{derpich2012improved}
showed that minimum the data-rate achievable by zero-delay source coding
of a Gaussian source subject to a quadratic distortion
constraint is closely approximated by the zero-delay rate-distortion
function (also called sequential- or non-anticipative rate-distortion
function in the literature).  For Gauss--Markov sources with
mean-square distortion criteria, computation of zero-delay
rate-distortion functions and construction of optimal test channels
are addressed by recent literature
\cite{stavrou2016filtering,tanaka2014semidefinite}.  Stavrou et
al. \cite{stavrou2016filtering} showed that the optimal test channel
can be realized by a memoryless Gaussian channel (sensor) with
feedback and a Kalman filter. Tanaka et
al. \cite{tanaka2014semidefinite} presented a different
realization of the test channel, using a memoryless Gaussian
channel without feedback and a Kalman filter.  The latter observation
implies that the zero-delay rate-distortion function can be computed
by considering the I-MMSE trade-off with respect to the Gaussian
channel gain (sensor gain matrix) \cite{tanaka2014semidefinite}.  The
I-MMSE trade-off in the present paper can be viewed as a
continuous-time counterpart of a similar trade-off considered in
discrete-time \cite{tanaka2014semidefinite}.  From analogous
discrete-time results, it is conjectured that results in this paper
provide fundamental performance limitations of zero-delay source
coding schemes for continuous-time sources, although zero-delay source
coding problems for continuous-time sources are not fully explored in
the literature.

This paper is organized as follows. Problem formulation is presented
in Section~\ref{secprob}. Section~\ref{secmain} summarizes the main
result. A connection to zero-delay source coding problem is discussed
in Section~\ref{secapp}. Section~\ref{seccon} summarizes the paper and
discuss future work.

Notation: Let $(\Omega, \mathcal{F}, \mathcal{P})$ be a probability
space and let $X$ be a random variable in a measurable space
$(\mathcal{X},\mathcal{A})$. The probability distribution $\mu_X$ of $X$
is defined by
\[
\mu_X(A)=\mathcal{P}\{\omega: X(\omega)\in A\}\;\; \forall A \in \mathcal{A}.
\]
If $X$ and $Y$ are random variables in the same measurable space with distributions $\mu_X$ and $\mu_Y$, the relative entropy from $Y$ to $X$ is defined by 
\[
D(\mu_X \| \mu_Y)=\int \log \frac{d\mu_X}{d \mu_Y} d\mu_X
\]
if the Radon-Nikodym derivative $\frac{d\mu_X}{d \mu_Y}$ exists.
If random variables $X$ and $Y$ have a joint probability distribution $\mu_{XY}$, the mutual information between $X$ and $Y$ is defined by
\[
I(X;Y)=D(\mu_{XY} \| \mu_X \otimes \mu_Y)
\]
where $\mu_X \otimes \mu_Y$ is the product measure defined by the
marginal  distributions. If $\mu_X$ is discrete,
the entropy of $X$ is defined by
\[
H(X)=-\sum_{x\in \mathcal{X}} \mu_X(x) \log \mu_X(x).
\]

\section{Problem Formulation}
\label{secprob}

Let $(\Omega, \mathcal{F}, \mathcal{P})$ be a complete probability space and $\mathcal{F}_t\subset \mathcal{F}$ be a non-decreasing family of $\sigma$-algebras. 
Let $(\bw_t, \mathcal{F}_t)$ and $(\bv_t, \mathcal{F}_t)$ be $n$-dimensional independent standard Wiener processes with respect to $\mathcal{P}$.
Assume that the source random process is an $n$-dimensional Gauss--Markov process of the form
\begin{equation}
\label{eqprocess}
d\bx_t=A\bx_tdt+B d\bw_t, \;\; t\in [0, \infty)
\end{equation}
with $\bx_0=0$. The source process is observed through an $n$-dimensional Gaussian channel (or sensor):
\begin{equation}
\label{eqobs}
d\by_t=C\bx_tdt+ d\bv_t, \;\; t\in [0, \infty)
\end{equation}
with $\by_0=0$. We assume that $(A,B)$ is a controllable pair.
\subsection{Minimum mean-square error (MMSE) estimate}
Let $\mathcal{F}_t^{\by} \subset \mathcal{F}$ be the  $\sigma$-algebra generated by $\by_s, 0\leq s\leq t$. 
Denote by $\hat{\bx}_t\triangleq \mathbb{E}(\bx_t|\mathcal{F}_t^{\by})$ the causal MMSE estimate of the process \eqref{eqprocess} via the observation \eqref{eqobs}, calculated by the Kalman--Bucy filter
\begin{equation}
\label{eqkbfilter}
d\hat{\bx}_t=A\hat{\bx}_tdt+P_tC^\top (d\by_t-C\hat{\bx}_tdt), \;\; t\in [0, \infty)
\end{equation}
with $\hat{\bx}_0=0$. In \eqref{eqkbfilter}, $P_t$ is the unique solution to the matrix Riccati differential equation
\begin{equation}
\label{eqriccati}
\frac{dP_t}{dt}=AP_t+P_tA^\top-P_tC^\top CP_t+BB^\top, \;\; t\in [0, \infty)
\end{equation}
with $P_0=0$.

For notational simplicity,  we denote by $\bx_0^T$, $\by_0^T$, $\hat{\bx}_0^T$ the random processes $\bx_t$, $\by_t$, $\hat{\bx}_t$ over the horizon $0\leq t \leq T$ as defined above. The MMSE performance over the considered horizon is denoted by
\[
\rho(\bx_0^T, \hat{\bx}_0^T)\triangleq \int_0^T \mathbb{E} \| \bx_t-\hat{\bx}_t\|^2 dt =\int_0^T \text{Tr} (P_t) \; dt.
\]

\subsection{Mutual information}
We are also interested in the mutual information $I(\bx_0^T;\hat{\bx}_0^T)$ between $\bx_0^T$ and $\hat{\bx}_0^T$. 

\begin{theorem}
\label{theoduncan}
Let the random processes $\bx_0^T$ and $\hat{\bx}_0^T$ be defined as above. Then
\[
I(\bx_0^T; \hat{\bx}_0^T)=\frac{1}{2}\int_0^T \mathbb{E}\|C (\bx_t-\hat{\bx}_t)\|^2 dt.
\]
\vspace{0ex}
\end{theorem}
\begin{proof}
Set $\bz_t=C\bx_t$. The following identity is well-known (e.g., Duncan \cite{duncan1970calculation}):
\begin{equation}
\label{eqduncan}
I(\by_0^T; \bz_0^T)=\frac{1}{2}\int_0^T \mathbb{E}\|C (\bx_t-\hat{\bx}_t)\|^2 dt.
\end{equation}
For completeness, a proof of \eqref{eqduncan} is given in Appendix A.
We also have an identity
\begin{equation}
\label{eqIXYZ}
I(Y_0^T;Z_0^T)= I(Y_0^T;X_0^T),
\end{equation}
whose proof is provided in Appendix B.

Due to the property of the Kalman--Bucy filter, $\hat{\bx}_0^T$ is a sufficient statistic of $\bx_0^T$ for $\by_0^T$. Thus
\begin{equation}
\label{eqsskalman}
I(\bx_0^T;\by_0^T)=I(\bx_0^T;\hat{\bx}_0^T).
\end{equation}
The claim follows from \eqref{eqduncan}--\eqref{eqsskalman}.
\end{proof}
It is immediate from Theorem~\ref{theoduncan} that the mutual information of interest can be written in terms of the solution $P_t$ to the Riccati equation \eqref{eqriccati} as
\[
I(\bx_0^T; \hat{\bx}_0^T)=\frac{1}{2}\int_0^T \text{Tr}(C P_t C^\top) dt.
\]

\subsection{I-MMSE trade-off via observation channel design}
In this paper, we construct the optimal observation gain $C\in \mathbb{R}^{n\times n}$ in the observation channel \eqref{eqobs} that minimizes the average mutual information  while the average MMSE is smaller than a given constant $D$. Formally, we seek  an optimal solution to the problem
\begin{subequations}
\label{optmain}
\begin{align}
R(D)\triangleq \inf_{C\in \mathcal{C}} \;\;\;& \limsup_{T\rightarrow +\infty} \frac{1}{T}I(\bx_0^T; \hat{\bx}_0^T) \\
\text{s.t.}\;\;\;\; & \limsup_{T\rightarrow +\infty}\frac{1}{T} \rho(\bx_0^T, \hat{\bx}_0^T)\leq D.
\end{align}
\end{subequations}
In \eqref{optmain}, the underlying linear system model \eqref{eqprocess} is given. The domain of optimization $\mathcal{C}\subset \mathbb{R}^{n\times n}$ is the set of matrices $C$ such that $(A,C)$ is a detectable pair, i.e., $A+LC$ is Hurwitz stable for some matrix $L$.
Below, we show that there exists an optimal solution and thus ``inf'' can be replaced by ``min.''

\section{Main Result}
\label{secmain}
We first assume that a precoder matrix $C\in\mathcal{C}$ is given.
Since we assume $(A,B)$ is controllable and $(A,C)$ is detectable, the algebraic Riccati equation
\begin{equation}
\label{eqARE}
AP+PA^\top-PC^\top CP+BB^\top=0
\end{equation}
admits a unique positive definite solution $P$ \cite[Theorem 13.7,
Corollary 13.8]{zhou1996robust}. Under the same assumption, the
solution $P_t$ to the Riccati differential equation \eqref{eqriccati}
with $P_0=0$ satisfies $P_t \rightarrow P$ as $t \rightarrow +\infty$
(e.g., \cite[Theorem 10.10]{bitmead1991riccati}), where $P$ is the
unique positive definite solution to \eqref{eqARE}.  Thus, it follows
from the convergence of Ces\`{a}ro mean that
\begin{align*}
\frac{1}{T}I(\by_0^T;\bz_0^T)&=\frac{1}{2T}\int_0^T \text{Tr}(CP_tC^\top)dt \rightarrow \frac{1}{2}\text{Tr}(CPC^\top) \\
\frac{1}{T}\rho(\bx_0^T,\hat{\bx}_0^T)&=\frac{1}{T}\int_0^T \text{Tr}(P_t)dt \rightarrow \text{Tr}(P) \;\;\text{ as }T\rightarrow +\infty.
\end{align*}
Hence, the right hand side of \eqref{optmain} can be written as
\begin{subequations}
\label{optCP}
\begin{align}
\inf_{C\in\mathcal{C}, P\succ 0} \;\; &\frac{1}{2}\text{Tr}(CPC^\top) \\
\text{s.t.} \;\;\;\;\;\; & AP+PA^\top-PC^\top CP+BB^\top=0  \label{eqricconst}\\
& \text{Tr}(P)\leq D.
\end{align}
\end{subequations}

Now we show that the optimization problem \eqref{optCP} is reformulated as a semidefinite programming problem.
First, under the equality constraint \eqref{eqricconst}, the objective function can be written as
\begin{subequations}
\label{eqstepA}
\begin{align}
\frac{1}{2}\text{Tr}(CPC^\top)&=\frac{1}{2}\text{Tr}(PC^\top C P P^{-1}) \label{eqstepA1}\\
&= \text{Tr}(A) +\frac{1}{2}\text{Tr}(B^\top P^{-1} B) \label{eqstepA2}\\
&= \min_{Q} \;\;\; \text{Tr}(A) +\frac{1}{2}\text{Tr}(Q) \label{eqstepA3}\\
& \hspace{3.5ex} \text{s.t.} \;\;\; B^\top P^{-1} B \preceq Q \nonumber \\
&= \min_{Q} \;\;\; \text{Tr}(A) +\frac{1}{2}\text{Tr}(Q) \label{eqstepA4}\\
& \hspace{3.5ex} \text{s.t.} \;\;\; \left[\begin{array}{cc}
Q & B^\top \!\!\\
B &\!\! P
\end{array}\right]\succeq 0. \nonumber
\end{align}
\end{subequations}
The equality constraint \eqref{eqricconst} is used to obtain \eqref{eqstepA2} from \eqref{eqstepA1}.
Equality \eqref{eqstepA3} holds since the unique solution to the minimization problem in \eqref{eqstepA3}  is $Q=B^\top P^{-1}B$. 
We have applied the Schur complement formula in \eqref{eqstepA4}.

The next lemma allows us to replace the nonlinear equality constraint \eqref{eqricconst} with a linear inequality constraint.
\begin{lemma}
\label{lemequiv}
If $(A,B)$ is controllable, then the following conditions are equivalent.
\begin{itemize}[leftmargin=3ex]
\item[(i)] $\exists C\in\mathcal{C}, P\succ 0 \text{ s.t. } AP+PA^\top-PC^\top CP+BB^\top=0.$
\item[(ii)] $\exists P\succ 0 \text{ s.t. } AP+PA^\top+BB^\top \succeq 0.$
\end{itemize}
\end{lemma}
\begin{proof}
The direction (i)$\Rightarrow$(ii) is trivial. To show (ii)$\Rightarrow$(i), notice that if condition (ii) holds, then clearly there exists a matrix $C$
such that 
\begin{equation}
\label{eqric1}
AP+PA^\top-PC^\top CP+BB^\top=0.
\end{equation}
To complete the proof, we show that for every $C$ satisfying \eqref{eqric1}, $(A,C)$ is a detectable pair. It is sufficient to show that $A-PC^\top C$ is stable.
To this end, rewrite \eqref{eqric1} as
\begin{equation}
\label{eqric2}
(A-PC^\top \!C)P+P(A-PC^\top \!C)^\top\!\!+PC^\top \!CP+BB^\top\!\!=\!0
\end{equation}
and suppose that $(A-PC^\top \!C)^\top$ is not stable. Let $\lambda$ be an unstable eigenvalue and $x$ be the corresponding eigenvector:
\begin{equation}
\label{eqevalvec}
(A^\top -C^\top CP)x=\lambda x.
\end{equation}
Pre- and post-multiplying \eqref{eqric2} by $x^*$ and $x$, we have
\[
(\lambda+\bar{\lambda})x^*Px+x^*(PC^\top CP+BB^\top)x=0.
\]
Since $\text{Re}(\lambda)\geq 0$ and $P\succ 0$, this implies $CPx=0$ and $B^\top x=0$. Thus, from \eqref{eqevalvec}, we obtain $A^\top x=\lambda x$ and $B^\top x=0$. This contradicts the Popov-Belevitch-Hautus (PBH) test for controllability of $(A,B)$.
\end{proof}
Applying \eqref{eqstepA} and Lemma~\ref{lemequiv} to \eqref{optCP}, we obtain the following result,
\begin{subequations}
\label{eqsdrinf}
\begin{align}
R(D)=\inf_{P\succ 0, Q} &\;\; \text{Tr}(A) +\frac{1}{2}\text{Tr}(Q) \\
\text{s.t.}\;\; & \; AP+PA^\top+BB^\top \succeq 0 \\
& \left[\begin{array}{cc}
Q & B^\top \!\!\\
B &\!\! P
\end{array}\right]\succeq 0 \\
&\;\; \text{Tr}(P)\leq D.
\end{align}
\end{subequations}
The main result of this paper is given by the next theorem.
\begin{theorem}
\label{theosdr}
Suppose $(A,B)$ is controllable. The optimal value $R(D)$ of \eqref{optmain} admits a semidefinite representation
\begin{align*}
R(D)=\min_{P\succ 0, Q} &\;\; \text{Tr}(A) +\frac{1}{2}\text{Tr}(Q) \\
\text{s.t.}\;\; & \; AP+PA^\top+BB^\top \succeq 0 \\
& \left[\begin{array}{cc}
Q & B^\top \!\!\\
B &\!\! P
\end{array}\right]\succeq 0 \\
&\;\; \text{Tr}(P)\leq D.
\end{align*}
In particular, there exists an optimal solution $P\succ 0$, $Q\succeq 0$ attaining the optimal value. Moreover, any matrix $C\in \mathcal{C}$ satisfying 
\[
AP+PA^\top-PC^\top CP+BB^\top=0,
\]
which always exists, is an optimal solution to \eqref{optmain}.
\end{theorem}
\begin{proof}
Since we have \eqref{eqsdrinf}, it is left to show that the optimal value is attained. By continuity, $\inf_{P\succ 0, Q}$ in \eqref{eqsdrinf} can be replaced with $\inf_{P\succeq 0, Q}$ without changing the optimal value. After this replacement, the existence of an optimal solution is guaranteed by Weierstrass' theorem \cite[Proposition A.8]{bertsekas1995nonlinear}, since the feasible domain for $(P,Q)$ is closed and the objective function is coercive. Thus  $\inf_{P\succeq 0, Q}$ can be written as $\min_{P\succeq 0, Q}$:
\begin{subequations}
\label{eqsdpsucceq}
\begin{align}
R(D)=\min_{P\succeq 0, Q} &\;\; \text{Tr}(A) +\frac{1}{2}\text{Tr}(Q) \label{eqsdpsucceq1} \\
\text{s.t.}\;\; & \; AP+PA^\top+BB^\top \succeq 0 \label{eqsdpsucceq2}\\
& \left[\begin{array}{cc}
Q & B^\top \!\!\\
B &\!\! P
\end{array}\right]\succeq 0 \label{eqsdpsucceq3}\\
&\;\; \text{Tr}(P)\leq D.
\end{align}
\end{subequations}
Now, we show that if $(P,Q)$ is an optimal solution to \eqref{eqsdpsucceq}, then $P$ is nonsingular. To show this by contradiction, assume $P\succeq 0$ is singular. Without loss of generality, assume
\[
P=\left[\begin{array}{cc} P_1 & 0 \\ 0 & 0\end{array}\right] \text{ with } P_1 \succ 0.
\]
Also, consider a corresponding partitioning of $A$ and $B$:
\[
A=\left[\begin{array}{cc} A_{11} & A_{12} \\ A_{21} & A_{22} \end{array}\right], B=\left[\begin{array}{c} B_1 \\ B_2 \end{array}\right].
\]

First, if $(P, Q)$ is a feasible solution to \eqref{eqsdpsucceq}, then it must be that $\text{Im}(B) \subseteq \text{Im}(P)$. To see this by contradiction, suppose there exists a matrix $N$ such that $NB\neq 0$ and $NP=0$. Then, pre- and post-multiplying   
\eqref{eqsdpsucceq3} by $\text{diag}(I, N)$ and $\text{diag}(I, N^\top)$, we obtain
\[
\left[\begin{array}{cc} Q\! & B^\top N^\top \\ NB &\! 0 \end{array}\!\right] \succeq 0.
\]
However, this is a contradiction because a matrix of this structure with non-zero off-diagonal entries must be indefinite. Thus, we conclude that $\text{Im}(B) \subseteq \text{Im}(P)$ and $B_2=0$.

Next, it must be that $A_{21}\neq 0$ since otherwise
\[
\text{dim}(\text{Im}[\;B \;\; AB \;\; A^2B \;\; ... \;\; A^{n-1}B\;])<n
\]
which contradicts the controllability of $(A,B)$.

Finally, with the above observations, \eqref{eqsdpsucceq1} becomes
\[
\left[\begin{array}{cc} A_{11}P_1+P_1A_{11}^\top+B_1B_1^\top \! & P_1A_{21}^\top \\
A_{21}P_1 & 0 \end{array}\!\right] \succeq 0
\]
which is again a contradiction since off-diagonal matrices are non-zero. Thus, we conclude that if $(P,Q)$ is an optimal solution to \eqref{eqsdpsucceq}, then $P$ is nonsingular.
\end{proof}

\section{Application}
\label{secapp}
\begin{figure}[t]
    \centering
    \includegraphics[width=\columnwidth]{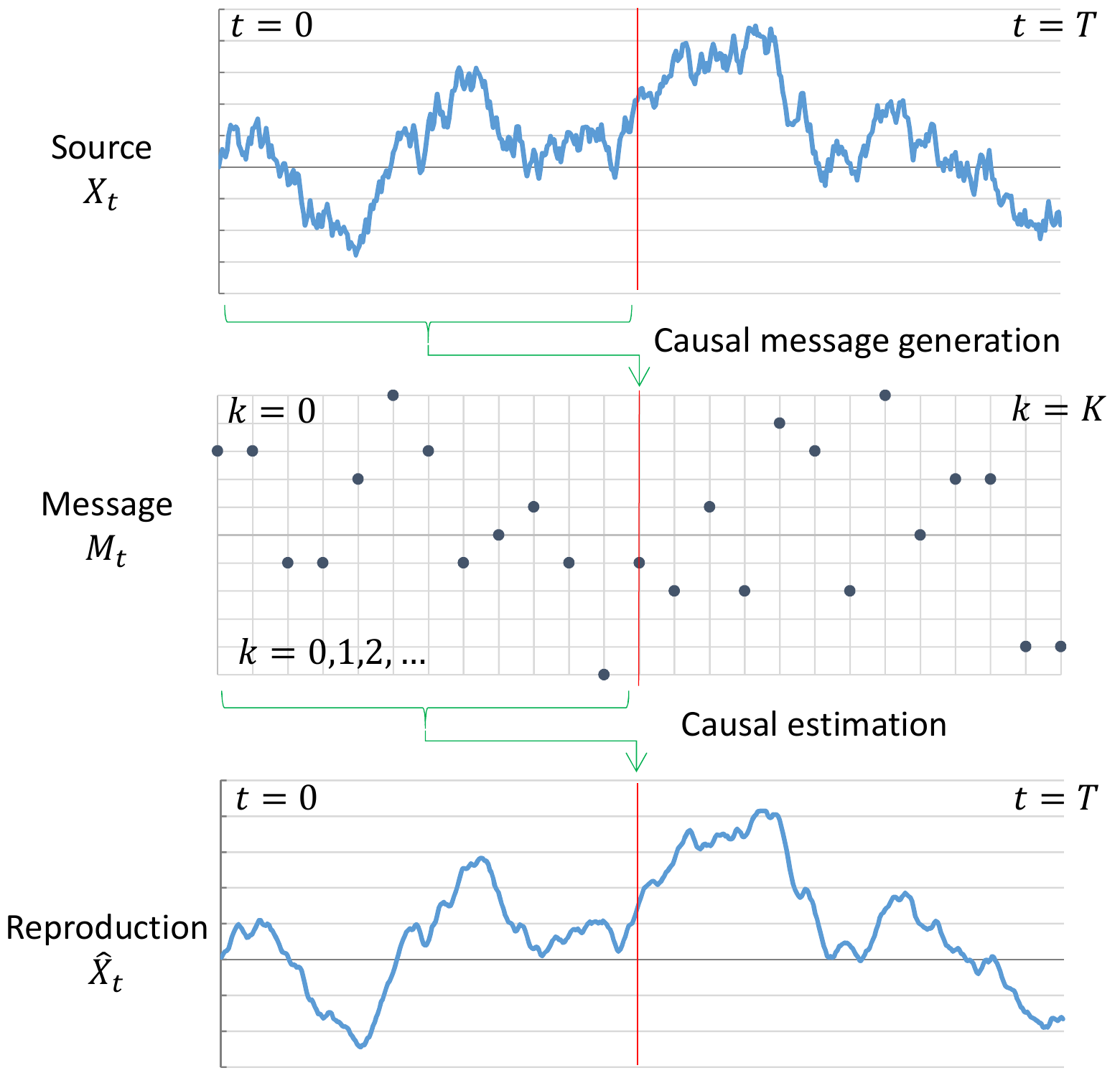}
    \caption{Zero delay source coding of continuous-time signal.}
    \label{fig:coding}
\end{figure}
In this section, we consider an application of the optimization problem \eqref{optmain} to a zero-delay source coding scenario depicted in Fig.~\ref{fig:coding}.

Let $X_t$ be a continuous-time source random process (e.g., video).
The source process is encoded with sampling period $\tau=\frac{T}{K}$, and a sequence of codewords $m_k=e_k(X_0^{k\tau})$, $k=1,
2, ... , K$ is generated. We assume $m_0=0$. For each $k=1,
2, ... , K$, we assume that $e_k$ is a $\mathbb{Z}^n$-valued map whose domain is the space of sample paths $X_t$, $0\leq t\leq k\tau$. A simple example is $n$ parallel scalar quantizers with uniform quantizer step sizes $\Delta=(\Delta_1, ... , \Delta_n)$:
\[
(m_k)_i=\lfloor \Delta_i (X_{k\tau})_i \rfloor, \forall i=1, 2, ... , n.
\]
Introduce a continuous-time process $M_t=m_{\lfloor \frac{t}{\tau} \rfloor}$ as the zero-order hold of $m_k$, and its time integral
\begin{equation}
\label{eqyZDSC}
Y_t=Y_0+\int_0^t M_s ds.
\end{equation}
At $t=k\tau$, $k=1, 2, ..., K$, the codeword $m_k$ is transmitted to the destination. At the destination, the decoder estimates $X_t$ in continuous-time based on the received information:
\[
\hat{X}_t=\mathbb{E}(X_t|M_0^t)=\mathbb{E}(X_t|Y_0^t).
\]
A function $\rho(X,\hat{X})=\int_0^T \mathbb{E}\|X_t-\hat{X}_t\|^2 dt$ is introduced as a distortion measure.

The above \emph{zero-delay source coding} scheme is denoted by $\mathsf{ZDSC}(\tau, e)$. (Notice that we are free to choose sampling period $\tau$ and encoding functions $e_1, ... , e_K$.)
Assuming that the source process is given by \eqref{eqprocess}, we are interested in the fundamental trade-off between the rate $\sum_{k=1}^K H(m_k)$ and the distortion $\rho(X_0^T,\hat{X}_0^T)$ achievable by $\mathsf{ZDSC}(\tau, e)$. Here, we are interested in the entropy $H(m_k)$ because it is related to the minimum expected codeword length if $m_k$ is represented by variable-length binary strings.

To analyze the fundamental performance limitation of $\mathsf{ZDSC}(\tau, e)$, we also consider a class of \emph{general causal reproduction} processes, denoted by $\mathsf{GCR}$, described below.
Let $(\Omega, \mathcal{F}, \mathcal{P})$ be a complete probability space and let $\mathcal{F}_t\subset \mathcal{F}$ be a non-decreasing family of $\sigma$-algebras. Let  $(W_t, \mathcal{F}_t)$ and $(V_t, \mathcal{F}_t)$ be mutually independent $n$-dimensional Wiener processes. Let the source process $X_t$ be defined by \eqref{eqprocess}.
Consider a random process $Y_t$ that can be represented by a stochastic integral
\begin{equation}
\label{eqyGZDSC}
Y_t=Y_0+\int_0^t M_s(X) ds + \int_0^t N_s(X)dV_s
\end{equation}
where for each $0 \leq t \leq T$, functions $M_t$ and $N_t$ are $\mathcal{F}_t^X$-measurable.
The source process is reproduced by
\[
\hat{X}_t=\mathbb{E}(X_t|Y_0^t).
\]
Notice that $\mathsf{ZDSC}(\tau, e)$ is a special case of $\mathsf{GCR}$ where \eqref{eqyGZDSC} has a special form \eqref{eqyZDSC}.

Notice that the following chain of inequalities holds.
\begin{subequations}
\begin{align}
&\min_{\mathsf{ZDSC}(\tau, e): \rho(X_0^T,\hat{X}_0^T)\leq D} \sum_{k=1}^T H(m_k) \label{chain1}\\
\geq & \min_{\mathsf{ZDSC}(\tau, e): \rho(X_0^T,\hat{X}_0^T)\leq D} H(m_1, ... , m_K) \label{chain2}\\
=& \min_{\mathsf{ZDSC}(\tau, e): \rho(X_0^T,\hat{X}_0^T)\leq D} H(Y_0^T) \label{chain3} \\
=& \min_{\mathsf{ZDSC}(\tau, e): \rho(X_0^T,\hat{X}_0^T)\leq D} I(X_0^T; Y_0^T) \label{chain4}\\
\geq & \min_{\mathsf{ZDSC}(\tau, e): \rho(X_0^T,\hat{X}_0^T)\leq D} I(X_0^T; \hat{X}_0^T) \label{chain5} \\
\geq & \min_{\mathsf{GCR}: \rho(X_0^T,\hat{X}_0^T)\leq D} I(X_0^T; \hat{X}_0^T) \label{chain6}
\end{align}
\end{subequations}
Equality \eqref{chain4} holds because $H(Y_0^T)=I(X_0^T; Y_0^T)+H(Y_0^T|X_0^T)$, and the second term is zero since under $\mathsf{ZDSC}(\tau, e)$ the map from $X_0^T$ to $Y_0^T$ is deterministic. \eqref{chain5} is the data-processing inequality. The last inequality \eqref{chain6} holds since $\mathsf{ZDSC}(\tau, e)$ is a special case of $\mathsf{GCR}$.
 
 Therefore, the smallest data rate that the zero delay source code can attain in average over the infinite horizon:
 \begin{align*}
R_{\mathsf{ZDSC}}(D)\triangleq \min_{\mathsf{ZDSC}(\tau, e)} & \limsup_{K\rightarrow +\infty} \frac{1}{K\tau}\sum_{k=1}^K H(m_k) \\
 \text{ s.t. } \;\;\;& \limsup_{K\rightarrow +\infty} \frac{1}{K\tau} \rho(\bx_0^{K\tau},\hat{\bx}_0^{K\tau}) \leq D
 \end{align*}
is lower bounded by
\begin{subequations}
\label{eqzdrdf}
 \begin{align}
R^*(D)\triangleq  \min_{\mathsf{GCR}} \;& \limsup_{T\rightarrow +\infty} \frac{1}{T}I(\bx_0^T;\hat{\bx}_0^T) \\
 \text{ s.t. } & \limsup_{T\rightarrow +\infty} \frac{1}{T} \rho(\bx_0^{T},\hat{\bx}_0^{T}) \leq D.
 \end{align}
 \end{subequations}
 Thus, we are interested in computing the function $R^*(D)$ since it
 provides a fundamental performance limitation for zero-delay
 source coding schemes.
 
Now, notice that the linear observation process \eqref{eqobs} is a special case of \eqref{eqyGZDSC}. Consequently, the functions $R(D)$ defined by \eqref{optmain} and $R^*(D)$ defined by \eqref{eqzdrdf} must satisfy
\begin{equation}
\label{eqzdrdfbound}
R^*(D)\leq R(D), \;\; \forall D >0.
\end{equation}
Since $R(D)$ is semidefinite representable (Theorem~\ref{theosdr}), computing $R(D)$ is straightforward. Unfortunately, the inequality \eqref{eqzdrdfbound} is not of great use because it only shows that $R(D)$ is an upper bound of a lower bound $R^*(D)$ of the smallest achievable data rate $R_{\mathsf{ZDSC}}(D)$. Nevertheless, guided by the analogy with the corresponding discrete-time results in \cite{tanaka2014semidefinite,srdstationary}, we conjecture that the inequality in (22) is actually the exact equality:
\begin{conjecture}
\label{conj1}
$R^*(D)= R(D), \;\; \forall D >0.$
\end{conjecture}
To establish Conjecture~\ref{conj1}, one essentially  needs to prove that the optimal observation process \eqref{eqyGZDSC} is linear in $\bx$, and has the form \eqref{eqobs}.

\section{Conclusion}
\label{seccon}
We considered a continuous-time vector Gauss--Markov process being
estimated by the Kalman--Bucy filter based on the observation through
a vector Gaussian channel (sensor). The trade-off between the mutual
information rate between the source process and the estimation process
and the MMSE, as well as trade-off achieving sensor gain matrices, are studied
by means of semidefinite programming. A connection to the zero-delay
rate-distortion problem is also discussed. In this paper, we restricted
ourselves to  observation through Gaussian
channels. However, in the future, it is worth pursuing further whether or not
the I-MMSE trade-off can be improved by considering non-Gaussian and
nonlinear sensor mechanisms (Conjecture~1).  Zero-delay source coding
schemes that (approximately) attain the obtained trade-off function
should also be considered in the future.




\section*{APPENDIX}
\subsection{Proof of equation \eqref{eqduncan}.}
\label{appA}

Let $\mathcal{C}^n=\left(C([0,T],\mathbb{R}^n),\mathcal{B}(C([0,T],\mathbb{R}^n))\right)$ be the measurable space of continuous functions $x=(x_t, t\in [0,T])$, $x:[0,T]\to
\mathbb{R}^n$ with $x_0=0$, equipped with the Borel $\sigma$-algebra
$\mathcal{B}_{\mathcal{C}^n}=\mathcal{B}(C([0,T],\mathbb{R}^n))$.
Consider two stochastic processes $Y=(Y_t, t\in [0,T])$ and  $Z=(Z_t, t\in [0,T])$ in a probability space $(\Omega, \mathcal{F}, P)$ related by
\begin{equation}
\label{eqyprocess}
d Y_t=Z_t dt+dV_t, \;\;\;Y_0=0
\end{equation}
where $V=(V_t, t\in [0,T])$ is the $n$-dimensional standard Brownian motion independent of $Z$.
Assume that $Z$ satisfies
\begin{equation}
\label{eqmsz}
\mathbb{E}\int_0^T \|Z_t\|^2 dt < \infty
\end{equation}
and $Z_0=0$.
Let $\mu_Y, \mu_V$ and $\mu_Z$ be probability measures on $\mathcal{C}^n$ defined by
\begin{align*}
\mu_Y(B_Y)&=P\{\omega: Y(\omega)\in B_Y\}, \;\; B_Y\in \mathcal{B}_{\mathcal{C}^n} \\
\mu_V(B_V)&=P\{\omega: V(\omega)\in B_V\}, \;\; B_V\in \mathcal{B}_{\mathcal{C}^n} \\
\mu_Z(B_Z)&=P\{\omega: Z(\omega)\in B_Z\}, \;\; B_Z\in \mathcal{B}_{\mathcal{C}^n}.
\end{align*}
In particular, $\mu_V$ is the Wiener measure. When $\mu_Y \ll \mu_V$, denote  by 
\[
\frac{d\mu_Y}{d\mu_V}: C[0,T] \rightarrow [0,\infty)
\]
the Radon-Nikodym derivative.

Let $\mu_{YZ}$ and $\mu_{VZ}$ be  joint measures on
$\mathcal{C}^n\times \mathcal{C}^n$ defined by the extensions of
\begin{align*}
\mu_{YZ}(B_Y\times B_Z)&=P\{\omega: Y(\omega)\in B_Y, Z(\omega)\in B_Z \}, \\
\mu_{VZ}(B_V\times B_Z)&=P\{\omega: V(\omega)\in B_V, Z(\omega)\in B_Z \},
\end{align*}
for $B_Y,B_V, B_Z\in
\mathcal{B}_{\mathcal{C}^n}$.
Since $V$ and $Z$ are independent,
$
\mu_{VZ}=\mu_V \otimes \mu_Z 
$
where $\mu_V \otimes \mu_Z$ is the product measure. Whenever $\mu_{YZ} \ll \mu_{VZ}$, denote by
\[
\frac{\mu_{YZ}}{\mu_{VZ}}: C[0,T]\times C[0,T] \rightarrow [0,\infty)
\]
the Radon-Nikodym derivative.

First, we derive an explicit formula for $\frac{d\mu_{YZ}}{\mu_{VZ}}$.
\begin{theorem}[Girsanov Theorem]
\cite[Theorem 6.3]{liptser2012statistics}:
\label{theogirsanov}
Let $\kappa=(\kappa_t, t \in[0,T])$ be a supermartingale of the form
\[
\kappa_t=\text{exp}\left(-\int_0^t Z_s^\top dV_s -\frac{1}{2}\int_0^t \|Z_s\|^2 ds \right)
\]
where $P(\int_0^T \|Z_t\|^2 dt<\infty)=1$. If $\mathbb{E}\kappa_T=1$, then the process $Y$ defined by \eqref{eqyprocess} is a Wiener process with respect to a probability measure $\tilde{P}$ such that $\frac{d\tilde{P}}{dP}=\kappa_T$.
\end{theorem}
\begin{proof}
See \cite[Theorem 6.3]{liptser2012statistics}.
\end{proof}

Note that condition \eqref{eqmsz} implies $P(\int_0^T \|Z_t\|^2 dt<\infty)=1$. To see this, consider
\begin{align*}
P\left(\int_0^T \|Z_t\|^2 dt<\infty\right)&\geq \sup_{R> 0} P\left(\int_0^T \|Z_t\|^2 dt\leq R\right) \\
&\geq  \sup_{R> 0} \left( 1-\frac{\mathbb{E}\int_0^T \|Z_t\|^2 dt}{R^2} \right) \\
&=1
\end{align*}
where the Chebyshev inequality is used in the second inequality. Moreover, since $Z$ and $V$ are independent, it follows that $\mathbb{E}\kappa_T=1$ \cite[Section 6.2, Example 4]{liptser2012statistics}.
Thus, premises of Theorem~\ref{theogirsanov} are satisfied. Condition \eqref{eqmsz} also implies $P(|\int_0^T Z_t^\top dV_t |<\infty)=1$. This can be verified as 
\begin{align*}
P\left(\left|\int_0^T Z_t^\top dV_t \right|<\infty\right)&=P\left(\left|\int_0^T Z_t^\top dV_t\right|^2<\infty\right) \\
&\geq \sup_{R> 0} P\left(\left|\int_0^T Z_t^\top dV_t \right|^2 \leq R \right) \\
&\geq  \sup_{R> 0} \left( 1-\frac{\mathbb{E}|\int_0^T Z_t^\top dV_t|^2}{R^2} \right) \\
&=  \sup_{R> 0} \left( 1-\frac{\mathbb{E}\int_0^T \|Z_t\|^2 dt}{R^2} \right) \\
&=1
\end{align*}
where the It\^o isometry \cite[Corollary 3.1.7]{oksendal2003stochastic} is used in the fourth line.
Hence $P(\kappa_T=0)=0$. Thus, by Theorem~\ref{theogirsanov}, together with \cite[Lemma 6.8]{liptser2012statistics}, we also have $P \ll \tilde{P}$ and $\frac{dP}{d\tilde{P}}=\kappa_T^{-1}$.
Now,
\begin{align}
\mu_{YZ}(B_Y\times B_Z)&=\int_{\{\omega:Y(\omega)\in B_Y, Z(\omega)\in B_Z\}} dP(\omega)  \nonumber \\
&=\int_{\{\omega:Y(\omega)\in B_Y, Z(\omega)\in B_Z\}} \!\!\! \kappa_T^{-1} d\tilde{P}(\omega) 
\end{align}
On the other hand, since $Y$ is a Wiener process under $\tilde{P}$, the joint probability distribution of $Y$ and $Z$ under $\tilde{P}$ is the same as the joint probability distribution of $V$ and $Z$ under $P$. Therefore,
\begin{align}
\frac{d\mu_{YZ}}{d\mu_{VZ}}(&Y(\omega), Z(\omega))=\kappa_T^{-1}(\omega) \nonumber \\
&=\text{exp}\left( \int_0^T Z_t^\top dV_t + \frac{1}{2}\int_0^T \|Z_t\|^2 dt \right) \nonumber \\
&=\text{exp}\left( \int_0^T Z_t^\top dY_t - \frac{1}{2}\int_0^T \|Z_t\|^2 dt \right). \label{eqmuyzmuvz}
\end{align}

Next, we derive an explicit formula for $\frac{d\mu_Y}{d\mu_V}$. 
\begin{theorem} 
\label{theo713}
\cite[Theorem 7.13]{liptser2012statistics}:
Let $\kappa=(\kappa_t, t\in[0,T])$ be a supermartingale of the form
\[
\kappa_t=\text{exp}\left(-\int_0^t Z_s^\top dV_s -\frac{1}{2}\int_0^t \|Z_s\|^2 ds \right)
\]
where $\int_0^T \mathbb{E}\|Z_t\| dt <\infty$ and $P(\int_0^T \|Z_t\|^2 dt<\infty)=1$. If $\mathbb{E}\kappa_T=1$, then $\mu_Y \ll \mu_V$, $\mu_V \ll \mu_Y$, and 
\[
\frac{d\mu_Y}{d\mu_V}(Y(\omega))=\text{exp}\left(\int_0^T \hat{Z}_t dY_t -\frac{1}{2}\int_0^T \|\hat{Z}_t\|^2 dt \right)
\]
where $\hat{Z}_t=\mathbb{E}(Z_t|\mathcal{F}_t^Y)$, $0\leq t\leq T$.
\vspace{0.3ex}
\end{theorem}
\begin{proof}
See \cite[Theorem 7.13]{liptser2012statistics}.
\end{proof}

Note that condition \eqref{eqmsz}, together with Jensen's inequality
\[
\left(\frac{1}{T}\int_0^T \mathbb{E} \|Z_t\|dt\right)^2 \leq \frac{1}{T}\int_0^T \mathbb{E}\|Z_t\|^2 dt
\]
implies $\int_0^T \mathbb{E} \|Z_t\|dt < \infty$. Thus, Theorem~\ref{theo713} is applicable, and
\begin{equation}
\label{eqmuvmuy}
\frac{d\mu_V}{d\mu_Y}(Y(\omega))=\text{exp}\left( -\int_0^T \hat{Z}_t^\top dY_t  +\frac{1}{2}\int_0^T \|\hat{Z}_t\|^2 dt\right).
\end{equation}

Finally, we derive an explicit formula for the mutual information 
\[
I(Y;Z)\triangleq \int_{\mathcal{C}^n \times \mathcal{C}^n} \log \frac{d\mu_{YZ}}{d(\mu_Y\otimes \mu_Z)}d\mu_{YZ}.
\]
Using the definition of Radon-Nikodym derivative, one can verify the chain rule
\[
\frac{d\mu_{YZ}}{d(\mu_V \otimes \mu_Z)}\frac{d\mu_V}{d\mu_Y}=\frac{d\mu_{YZ}}{d(\mu_Y\otimes \mu_Z)}.
\] 
Thus
\begin{align}
&\log \frac{d\mu_{YZ}}{d(\mu_Y\otimes \mu_Z)} \nonumber \\
&=\log \frac{d\mu_{YZ}}{d\mu_{VZ}} + \log \frac{d\mu_V}{d\mu_Y} \nonumber \\
&=\int_0^T (Z_t-\hat{Z}_t)^\top dY_t-\frac{1}{2}\int_0^T (\|Z_t\|^2-\|\hat{Z}_t\|^2) dt \label{eqlog1}\\
&=\int_0^T (Z_t-\hat{Z}_t)^\top dV_t+\frac{1}{2}\int_0^T \|Z_t- \hat{Z}_t\|^2 dt. \label{eqlog2}
\end{align}
Equations  \eqref{eqmuyzmuvz} and \eqref{eqmuvmuy} are used in \eqref{eqlog1}. 
Taking the expectation, the first term in \eqref{eqlog2} vanishes \cite[Theorem 3.2.1]{oksendal2003stochastic}.
Thus,
\[
I(Y;Z)=\mathbb{E}\log \frac{d\mu_{YZ}}{d(\mu_Y\otimes \mu_Z)}=\frac{1}{2}\int_0^T \mathbb{E}\|Z_t-\hat{Z}_t\|^2 dt. \]

\subsection{Proof of equation \eqref{eqIXYZ}}
\label{appB}

Let $(\mathcal{C}^n, \mathcal{B}_{\mathcal{C}^n})$ be the measurable space of continuous functions as defined in Appendix A.
Consider stochastic processes $X=(X_t,t\in [0,T])$ and $Y=(Y_t, 
t\in [0,T])$ defined in (1), (2) and $Z_t=CX_t$. Since $X$ is an Ito
process, its trajectory is a.s. continuous, and so are trajectories of
$Y$ and $Z$. This allows us to define measures $\mu_X$, $\mu_Y$, $\mu_Z$ on $\mathcal{C}^n$ by 
\begin{align*}
\mu_X(B_X)&=P\{\omega:X_0^T(\omega)\in B_X\} \\
\mu_Y(B_Y)&=P\{\omega:Y_0^T(\omega)\in B_Y\} \\
\mu_Z(B_Z)&=P\{\omega:Z_0^T(\omega)\in B_Z\} =P\{\omega:CX_0^T(\omega)\in B_Z\}
\end{align*}
where $B_X, B_Y, B_Z\in \mathcal{B}_{\mathcal{C}^n}$.
Consider a mapping $c: \mathcal{C}^n\rightarrow \mathcal{C}^n$ defined by $z=Cx$. For each $B_Z\in \mathcal{B}_{\mathcal{C}^n}$, define $c^{-1}(B_Z)\in \mathcal{B}_{\mathcal{C}^n}$ by $c^{-1}(B_Z)\triangleq \{x\in \mathcal{C}^n: Cx\in B_Z\}$.

\begin{claim}
\label{claimmuxz}
$
\mu_Z(B_Z)=\mu_X(c^{-1}B_Z) \quad \forall B_Z\in \mathcal{B}_{\mathcal{C}^n}.
$
\end{claim}
\begin{proof}
Let
 $\omega$ be such that $X_0^T(\omega)\in c^{-1}(B_Z)$, then we have that
 $CX_0^T(\omega)\in B_Z$. Therefore, $\{\omega: X_0^T(\omega)\in
 c^{-1}(B_Z)\}\subseteq \{\omega: CX_0^T(\omega)\in B_Z\}$ and
 $\mu_X(c^{-1}(B_Z))\le \mu_Z(B_Z)$. 
 On the other hand,
since $X_0^T$ is an Ito process, it is
a.s. continuous~\cite{liptser2012statistics}, therefore
$P\{\omega:X_0^T(\omega)\not\in
          C([0,T],\mathbb{R}^n)\}=0$. This allows us to conclude that
\begin{align*}
&\mu_Z(B_Z) \\
&=P\{\omega: CX_0^T(\omega)\in B_Z\} \\
&=P\{\omega: CX_0^T(\omega)\in B_Z, X_0^T(\omega)\in C([0,T],\mathbb{R}^n)\} \\
&\hspace{3ex}+P\{\omega: CX_0^T(\omega)\in B_Z, X_0^T(\omega)\not\in C([0,T],\mathbb{R}^n)\}  \\
&\leq P\{\omega: CX_0^T(\omega)\in B_Z, X_0^T(\omega)\in C([0,T],\mathbb{R}^n)\} \\
&\hspace{3ex}+P\{\omega: X_0^T(\omega)\not\in C([0,T],\mathbb{R}^n)\}  \\
&=P\{\omega: CX_0^T(\omega)\in B_Z, X_0^T(\omega)\in C([0,T],\mathbb{R}^n)\} \\
&=\mu_X(c^{-1}(B_Z))
\end{align*} 
Thus the claim holds. 
\end{proof}

\begin{claim}
\label{claimintmu}
For any measurable function $f$, 
\[
\int_{\mathcal{C}^n}
f(z)\mu_Z(dz)=\int_{\mathcal{C}^n} f(Cx)\mu_X(dx). 
\]
\end{claim}
\vspace{2ex}
\begin{proof}
Notice that
\begin{align*}
\int_{\mathcal{C}^n} f(z)\mu_Z(dz)&=\int_{c^{-1}(\mathcal{C}^n)} f(Cx)\mu_X(dx) \\
&=\int_{\mathcal{C}^n} f(Cx)\mu_X(dx)
\end{align*}
The first equality is a consequence of Claim~\ref{claimmuxz}. 
The second equality holds since $c^{-1}(\mathcal{C}^n)=\mathcal{C}^n$. To see this, notice by definition
$
c^{-1}(\mathcal{C}^n)\triangleq \{x\in\mathcal{C}^n: Cx \in \mathcal{C}^n\} \subseteq \mathcal{C}^n$.
Conversely, $\mathcal{C}^n \subseteq c^{-1}(\mathcal{C}^n)$ since $Cx$ is continuous for any continuous function $x$.
\end{proof}

Now we prove equation \eqref{eqIXYZ}.

\begin{lemma}\label{DPI}
$
I(Y_0^T;Z_0^T)= I(Y_0^T;X_0^T)$.
\end{lemma}

\begin{proof}
In addition to $\mu_X$, $\mu_Y$, $\mu_Z$, consider the measures
$\mu_{YX}$ and $\mu_{YZ}$  
on the product space $\mathcal{C}^n \times \mathcal{C}^n$ defined by the extensions of 
\begin{align*}
\mu_{YX}(B_Y \times B_X)&=P\{\omega: Y_0^T(\omega)\in B_Y, X_0^T(\omega)\in B_X\}\\
\mu_{YZ}(B_Y \times B_Z)&=P\{\omega: Y_0^T(\omega)\in B_Y, Z_0^T(\omega)\in B_Z\}.
\end{align*}
where $B_X,B_Y, B_Z\in
\mathcal{B}_{\mathcal{C}^n}$.
Since $\mathcal{C}^n$ is a Borel space~\cite[Definition
7.7]{bertsekas2004stochastic}, by \cite[Theorem~5.1.9 and Exercise~5.1.16]{Durrett-2010}
there exists a Borel-measurable stochastic kernel $\mu_{Y|X}$
on $\mathcal{C}^n$ given $\mathcal{C}^n$, such that 
$\mu_{Y|X}(B_Y|X_0^T(\omega))$ is a 
version of $P(\{\omega:Y_0^T(\omega)\in
B_Y\}|\mathcal{B}^{X_0^T})$; here $\mathcal{B}^{X_0^T}$ denotes the
$\sigma$-algebra of events generated by $X_0^T$. That is, the regular
conditional probability distribution given $\mathcal{B}^{X_0^T}$ exists
and   
\begin{eqnarray}
\lefteqn{\mu_{YX}(B_Y\times B_X)=P\{\omega: Y_0^T(\omega)\in B_Y, X_0^T(\omega)\in
B_X\}} &&  \nonumber \\
&=&\int_{\{\omega: X_0^T(\omega)\in B_X\}} P(Y_0^T\in
B_Y|\mathcal{B}^{X_0^T})P(d\omega) \nonumber \\
&=&\int_{B_X} P\left(\int_0^{(\cdot)}Cx_sds+V\in
B_Y\right)P(\omega: X_0^T\in dx) \nonumber \\
&=&\int_{B_X} \mu_{Y|X}(B_Y|x) \mu_X(dx).
\label{muXY}
\end{eqnarray}
The identity in the second line follows from the existence of the  regular
conditional probability distribution given $\mathcal{B}^{X_0^T}$, and the
identity in the third line is due to the change of variables. Here we have used
the notation $\int_0^{(\cdot)}Cx_sds+V$ to stress that 
we consider the entire path of the random process $F_{x,t}=\int_0^{t}Cx_s
ds+V_t$ parameterized by $x\in \mathcal{C}^n$. The last line in (\ref{muXY})
holds due to the 
uniqueness of the Radon-Nikodym derivative. This leads us to conclude,
that for almost all $x\in \mathcal{C}^n$,
\[
\mu_{Y|X}(B_Y|x)= P\left(\omega: F_{x,0}^T\in B_Y\right) \quad \forall B_Y\in\mathcal{B}_{\mathcal{C}^n}. 
\]
In a similar fashion, it follows that there exists a Borel-measurable
stochastic kernel $\mu_{Y|Z}$ on $\mathcal{C}^n$ given $\mathcal{C}^n$, such that 
$\mu_{Y|Z}(B_Y|Z_0^T(\omega))$ is a 
version of $P(\{\omega:Y_0^T(\omega)\in
B_Y\}|\mathcal{B}^{Z_0^T})$ and for almost all $z\in\mathcal{C}^n$,
\[
\mu_{Y|Z}(B_Y|z)= P\left(\omega: G_{z,0}^T\in B_Y\right)\quad \forall B_Y\in\mathcal{B}_{\mathcal{C}^n}, 
\]
where $G_{z,t}=\int_0^{t}z_s ds+V_t$. It is also clear from these
expressions that 
\begin{equation}
\mu_{Y|X}(B_Y|x)=\mu_{Y|Z}(B_Y|Cx).
\label{muYZ=muYX}
\end{equation}

We are now in  a position to complete the proof. By definition of the
mutual information, 
\begin{align*}
 I(Y_0^T;X_0^T) &= D(\mu_{YX}\|\mu_Y\otimes \mu_X),  \\
 I(Y_0^T;Z_0^T) &= D(\mu_{YZ}\|\mu_Y\otimes \mu_Z).
\end{align*}
Thus, the result follows from the chain of equalities:
\begin{subequations}
\begin{align}
&D(\mu_{YZ}\|\mu_Y\otimes \mu_Z) \nonumber \\
&=
\int_\mathcal{Z} D(\mu_{Y|Z}(\cdot|z)\|\mu_Y(\cdot))\mu_Z(dz) \label{eqchain1}
\\ 
&=
\int_\mathcal{X} D(\mu_{Y|Z}(\cdot|Cx)\|\mu_Y(\cdot))\mu_X(dx) \label{eqchain3}
\\
&=
\int_\mathcal{X} D(\mu_{Y|X}(\cdot|x)\|\mu_Y(\cdot))\mu_X(dx) \label{eqchain4}\\
&=
D(\mu_{YX}\|\mu_Y\otimes \mu_X).\label{eqchain5}
\end{align}
\end{subequations}
Equalities \eqref{eqchain1} and \eqref{eqchain5} follow from the chain rule of relative entropy \cite[Lemma 1.4.3(f)]{dupuis2011weak}. 
Equation \eqref{eqchain3} follows from Claim~\ref{claimintmu}, and 
\eqref{eqchain4} follows from (\ref{muYZ=muYX}).
\end{proof}


\bibliographystyle{ieeetran}
\bibliography{ref}

\begin{thebibliography}{10}
\providecommand{\url}[1]{#1}
\csname url@samestyle\endcsname
\providecommand{\newblock}{\relax}
\providecommand{\bibinfo}[2]{#2}
\providecommand{\BIBentrySTDinterwordspacing}{\spaceskip=0pt\relax}
\providecommand{\BIBentryALTinterwordstretchfactor}{4}
\providecommand{\BIBentryALTinterwordspacing}{\spaceskip=\fontdimen2\font plus
\BIBentryALTinterwordstretchfactor\fontdimen3\font minus
  \fontdimen4\font\relax}
\providecommand{\BIBforeignlanguage}[2]{{%
\expandafter\ifx\csname l@#1\endcsname\relax
\typeout{** WARNING: IEEEtran.bst: No hyphenation pattern has been}%
\typeout{** loaded for the language `#1'. Using the pattern for}%
\typeout{** the default language instead.}%
\else
\language=\csname l@#1\endcsname
\fi
#2}}
\providecommand{\BIBdecl}{\relax}
\BIBdecl

\bibitem{guo2005mutual}
D.~Guo, S.~Shamai, and S.~Verd{\'u}, ``Mutual information and minimum
  mean-square error in {G}aussian channels,'' \emph{IEEE Transactions on
  Information Theory}, vol.~51, no.~4, pp. 1261--1282, 2005.

\bibitem{duncan1970calculation}
T.~E. Duncan, ``On the calculation of mutual information,'' \emph{SIAM Journal
  on Applied Mathematics}, vol.~19, no.~1, pp. 215--220, 1970.

\bibitem{kadota1971mutual}
T.~Kadota, M.~Zakai, and J.~Ziv, ``Mutual information of the white {G}aussian
  channel with and without feedback,'' \emph{IEEE Transactions on Information
  theory}, vol.~17, no.~4, pp. 368--371, 1971.

\bibitem{weissman2013directed}
T.~Weissman, Y.-H. Kim, and H.~H. Permuter, ``Directed information, causal
  estimation, and communication in continuous time,'' \emph{IEEE Transactions
  on Information Theory}, vol.~59, no.~3, pp. 1271--1287, 2013.

\bibitem{guo2008mutual}
D.~Guo, S.~Shamai, and S.~Verd{\'u}, ``Mutual information and conditional mean
  estimation in {P}oisson channels,'' \emph{IEEE Transactions on Information
  Theory}, vol.~54, no.~5, pp. 1837--1849, 2008.

\bibitem{atar2012mutual}
R.~Atar and T.~Weissman, ``Mutual information, relative entropy, and estimation
  in the {P}oisson channel,'' \emph{IEEE Transactions on Information theory},
  vol.~58, no.~3, pp. 1302--1318, 2012.

\bibitem{kabanov1978capacity}
Y.~M. Kabanov, ``The capacity of a channel of the {P}oisson type,''
  \emph{Theory of Probability \& Its Applications}, vol.~23, no.~1, pp.
  143--147, 1978.

\bibitem{jiao2016mutual}
J.~Jiao, K.~Venkat, and T.~Weissman, ``Mutual information, relative entropy and
  estimation error in semi-martingale channels,'' \emph{Information Theory
  (ISIT), 2016 IEEE International Symposium on}, pp. 2794--2798, 2016.

\bibitem{palomar2006gradient}
D.~P. Palomar and S.~Verd{\'u}, ``Gradient of mutual information in linear
  vector {G}aussian channels,'' \emph{IEEE Transactions on Information Theory},
  vol.~52, no.~1, pp. 141--154, 2006.

\bibitem{derpich2012improved}
M.~S. Derpich and J.~{\O}stergaard, ``Improved upper bounds to the causal
  quadratic rate--distortion function for {G}aussian stationary sources,''
  \emph{IEEE Transactions on Information Theory}, vol.~58, no.~5, pp.
  3131--3152, 2012.

\bibitem{stavrou2016filtering}
P.~A. Stavrou, T.~Charalambous, and C.~D. Charalambous, ``Filtering with
  fidelity for time-varying {G}auss--{M}arkov processes,'' \emph{The 55th IEEE
  Conference on Decision and Control (CDC)}, pp. 5465--5470, 2016.

\bibitem{tanaka2014semidefinite}
T.~Tanaka, K.-K. Kim, P.~A. Parrilo, and S.~K. Mitter, ``Semidefinite
  programming approach to {G}aussian sequential rate--distortion trade-offs,''
  \emph{IEEE Transactions on Automatic Control (To appear)}, 2014.

\bibitem{zhou1996robust}
K.~Zhou, J.~C. Doyle, and K.~Glover, \emph{Robust and optimal control}.\hskip
  1em plus 0.5em minus 0.4em\relax Prentice hall New Jersey, 1996.

\bibitem{bitmead1991riccati}
R.~R. Bitmead and M.~Gevers, ``Riccati difference and differential equations:
  Convergence, monotonicity and stability,'' in \emph{The Riccati
  Equation}.\hskip 1em plus 0.5em minus 0.4em\relax Springer, 1991, pp.
  263--291.

\bibitem{bertsekas1995nonlinear}
D.~Bertsekas, \emph{Nonlinear Programming}.\hskip 1em plus 0.5em minus
  0.4em\relax Athena Scientific, 1995.

\bibitem{srdstationary}
T.~Tanaka, ``Semidefinite representation of sequential rate--distortion
  function for stationary {G}auss--{M}arkov processes,'' \emph{The 2015 IEEE
  Multi-Conference on Systems and Control (MSC)}, 2015.

\bibitem{liptser2012statistics}
R.~Liptser and A.~Shiryaev, \emph{Statistics of Random Processes I: General
  Theory}.\hskip 1em plus 0.5em minus 0.4em\relax Springer, 2012.

\bibitem{oksendal2003stochastic}
B.~{\O}ksendal, \emph{Stochastic differential equations}.\hskip 1em plus 0.5em
  minus 0.4em\relax Springer, 2003.

\bibitem{bertsekas2004stochastic}
D.~P. Bertsekas and S.~Shreve, \emph{Stochastic optimal control: the
  discrete-time case}, 2004.

\bibitem{Durrett-2010}
R.~Durrett, \emph{Probability: Theory and Examples}, 4th~ed.\hskip 1em plus
  0.5em minus 0.4em\relax Cambridge University press, 2010.

\bibitem{dupuis2011weak}
P.~Dupuis and R.~S. Ellis, \emph{A weak convergence approach to the theory of
  large deviations}.\hskip 1em plus 0.5em minus 0.4em\relax John Wiley \& Sons,
  2011, vol. 902.

\end{thebibliography}

\end{document}